\documentclass[10pt,a4paper]{amsart}

\usepackage{graphicx,amssymb,color,amscd}

\theoremstyle{plain}
\newtheorem{theorem}{Theorem}[section]
\newtheorem{lemma}[theorem]{Lemma}
\newtheorem{proposition}[theorem]{Proposition}
\newtheorem{corollary}[theorem]{Corollary}
\theoremstyle{definition}

\theoremstyle{remark}
\newtheorem{convention}[theorem]{Convention} 
\newtheorem{remark}[theorem]{Remark}
\newtheorem{fact}[theorem]{Fact}
\newtheorem{claim}[theorem]{Claim}
\newtheorem*{acknowledgments}{Acknowledgments}
\numberwithin{equation}{section}
\numberwithin{figure}{section}

\newcommand{\bd}{\begin{description}}   
\newcommand{\ed}{\end{description}} 
\newcommand{\ba}{\begin{array}}      \newcommand{\ea}{\end{array}} 
\newcommand{\bc}{\begin{center}}     \newcommand{\ec}{\end{center}} 
\newcommand{\be}{\begin{enumerate}}  \newcommand{\ee}{\end{enumerate}} 
\newcommand{\beq}{\begin{eqnarray}}  \newcommand{\eeq}{\end{eqnarray}} 
\newcommand{\beQ}{\begin{eqnarray*}} \newcommand{\eeQ}{\end{eqnarray*}} 
\newcommand{\bi}{\begin{itemize}}    \newcommand{\ei}{\end{itemize}}

\newcommand{\nbpt}{18}
\newcommand{\figtotext}[3]{\begin{array}{c}\includegraphics{#3}\end{array}}
\newcommand{\double}{\figtotext{\nbpt}{\nbpt}{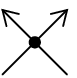}}
\newcommand{\Over}{\figtotext{\nbpt}{\nbpt}{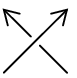}}
\newcommand{\under}{\figtotext{\nbpt}{\nbpt}{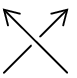}}

\newcommand{\si}{\sigma}
\newcommand{\1}{\mathbf{1}}

\begin{document} 
\title[Abelian quotients of the string link monoid]{Abelian quotients of the string link monoid} 
\author[J.B. Meilhan]{Jean-Baptiste Meilhan} 
\address{Institut Fourier, Universit\'e Grenoble 1 \\
         100 rue des Maths - BP 74\\
         38402 St Martin d'H\`eres , France}
	 \email{jean-baptiste.meilhan@ujf-grenoble.fr}
\author[A. Yasuhara]{Akira Yasuhara} 
\address{Tokyo Gakugei University\\
         Department of Mathematics\\
         Koganeishi \\
         Tokyo 184-8501, Japan}
	 \email{yasuhara@u-gakugei.ac.jp}
%
%
%
\begin{abstract} 
The set $\mathcal{SL}(n)$ of $n$-string links has a monoid structure, given by the stacking product.
When considered up to concordance, $\mathcal{SL}(n)$ becomes a group, which is known to be abelian 
only if $n=1$. 
In this paper, we consider two families of equivalence relations which endow $\mathcal{SL}(n)$ with a group structure, 
namely the $C_k$-equivalence introduced by Habiro in connection with finite type invariants theory, 
and the $C_k$-concordance, which is generated by $C_k$-equivalence and concordance. 
We investigate under which condition these groups are abelian, and give applications to finite type invariants. 
\end{abstract} 

\maketitle
\section{Introduction}
For a positive integer $n$, let $D^2$ be the standard two-dimensional disk equipped with $n$ marked points $x_1,...,x_n$ in its interior.  
Let $I$ denote the unit interval.  
An \textit{$n$-string link} is a proper embedding  
\[ L : \bigsqcup_{i=1}^n I_i \rightarrow D^2\times I, \]
of the disjoint union $\sqcup_{i=1}^{n} I_i$ of $n$ copies of $I$ in $D^2\times I$, 
such that for each $i$, the image $L_i$ of $I_i$ runs from $(x_i,0)$ to $(x_i,1)$.   
Abusing notation, we will also denote by $L \subset D^2\times I$ the image of the map $L$, 
and $L_i$ is called the $i$th string of $L$.  
Note that each string of an $n$-string link is equipped with an (upward) orientation 
induced by the natural orientation of $I$.
The string link $\sqcup_{i=1}^n(\{x_i\}\times I)$ is called the {\em trivial $n$-string link} and is
denoted by $\1_n$, or simply by $\1$.

The set $\mathcal{SL}(n)$ of isotopy classes of $n$-string links fixing the endpoints has a monoid structure, 
with composition given by the \emph{stacking product} and with the trivial $n$-string link $\1_n$ as unit element.   
There is a surjective map from $\mathcal{SL}(n)$ to the set of isotopy classes of $n$-component links,
which sends an $n$-string link to its closure (in the usual sense).  
For $n=1$, this map is a monoid isomorphism.  

It is well known that the monoid $\mathcal{SL}(n)$ is not a group. In fact, it is quite far from being a group : 
the group of invertible elements in $\mathcal{SL}(n)$ is actually the pure braid group on $n$ strands \cite{HL}. 
However, $\mathcal{SL}(n)$ becomes a group when considered up to concordance. 
Recall that two $n$-string links $L, L'$ are \emph{concordant} if there is an embedding 
$$ f: \left(\sqcup_{i=1}^n I_i\right)\times I \longrightarrow \left(D^2\times I\right)\times I $$  
such that 
$f\left( (\sqcup_{i=1}^n I_i)\times \{ 0 \} \right)=L$ and $f\left( (\sqcup_{i=1}^n I_i)\times \{ 1 \} \right)=L'$, 
and such that $f\left(\partial(\sqcup_{i=1}^n I_i)\times I\right)=(\partial L) \times I$.   
The inverse of a string link up to concordance is simply its horizontal mirror image with reversed orientation. 
Le Dimet showed that the group of concordance classes of $n$-string links is not abelian for $n\ge 3$ in \cite{ledimet}.\footnote{
The literature sometimes erroneously refers to \cite{ledimet} for the analogous fact for $n=2$, 
but, as Le Dimet writes in \cite[4.5 Conclusions 2.]{ledimet}, his arguments did not allow him to conclude in the case of $2$-string links. }
The fact that this result also hold for $n=2$ seem to have been first observed by De Campos \cite{decampos}, 
as a consequence of a result of Miyazaki for the theta-curve cobordism group \cite{miyazaki}. 

Now, other equivalence relations are known to endow $\mathcal{SL}(n)$ with a group structure:  
for each $k\ge 1$, the set $\mathcal{SL}(n)/C_k$ of \emph{$C_k$-equivalence} classes of $n$-string links is a group \cite{H}. 
Here, the $C_k$-equivalence is an equivalence relation on links generated by \emph{$C_k$-moves} and ambient isotopies, defined in connection 
with the theory of finite type invariants. 
A $C_1$-move is just a crossing change, and for any integer $k\ge 2$, a $C_k$-move is a local move on links as illustrated in Figure \ref{cnm}. 
\begin{figure}[!h]
\includegraphics{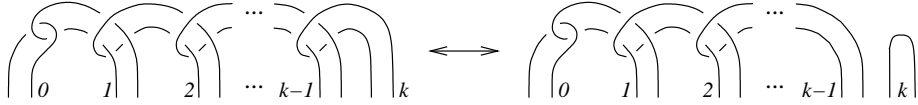}
\caption{A $C_k$-move involves $k+1$ strands of a link, labelled here by integers between $0$ and $k$.   } \label{cnm}
\end{figure}
These local moves can also be defined in terms of `insertion' of an element of the $k$th lower central  
subgroup of some pure braid group \cite{stanford}, or alternatively by using the theory of claspers (see Section \ref{claspers}). 

In this paper, we investigate under which condition the quotient group $\mathcal{SL}(n)/C_k$ is abelian. 
This is immediate, for $k=1$ since $\mathcal{SL}(n)/C_1$ is the trivial group. 
It is also well-known that, for all $n\ge 1$, the group $\mathcal{SL}(n)/C_2$ is abelian. 
This essentially follows from the fact, due to Murakami and Nakanishi \cite{MN}, 
that $C_2$-equivalence classes of (string) links are classified by the linking number. 
Another rather easy fact is the following, which uses Milnor's triple linking number (see Section \ref{sec:proof0}).
\begin{proposition}\label{nonabeliansl}
The group $\mathcal{SL}(n)/C_k$ of $C_k$-equivalence classes of $n$-string links is not abelian for any $n\ge 3$ and any $k\ge 3$. 
\end{proposition}
Hence our study is reduced to the case of $2$-string links. 
In \cite{akira}, the second author showed that $\mathcal{SL}(2)/C_k$ is not abelian for $k\geq 12$. 
Here we improve this result and give an almost complete answer for this \lq abelian problem'. 
\begin{theorem}\label{abeliansl}
(1).~The group $\mathcal{SL}(2)/C_k$ of $C_k$-equivalence classes of $2$-string links is abelian for $k<7$. \\
(2).~The group $\mathcal{SL}(2)/C_k$ is \emph{not} abelian for $k>7$. 
\end{theorem}
Although the case $k=7$ remains open so far, we show the following.  
\begin{theorem}\label{prop:c7}
 The group $\mathcal{SL}(2)/C_7$ is abelian if it has no $2$-torsion.
\end{theorem}
\begin{remark}
Bar-Natan's computations \cite{BNcalc} shows that $7$ is the smallest degree of a finite type invariant 
that can detect the orientation of $2$-string links. 
The existence of such an invariant was shown by Duzhin and Karev \cite{DK}, and this result is used in 
Section \ref{sec:proof2} to prove Theorem \ref{abeliansl} (2). 
Likewise, it would be very interesting to determine whether there exists a $\mathbb{Z}_2$-valued finite type invariant of degree $6$ 
that can detect the orientation of $2$-string links. 
\end{remark}

The `abelian problem' addressed above is deeply related to one of the main results in the theory of finite type invariants, 
due to Habiro \cite{H} and Goussarov \cite{G} independently, 
which gives a topological characterization of the informations contained by finite type invariants of \emph{knots}. 
\begin{theorem}[\cite{G,H}]\label{cnknots}
Two knots cannot be distinguished by any finite type invariant of order less than $k$ if and only if they 
are $C_k$-equivalent. 
\end{theorem}  
\noindent It was indeed conjectured by both Goussarov and Habiro that, although it fails to hold for links, 
Theorem \ref{cnknots} may generalize to string links.  
One of the key ingredient of Habiro's proof of Theorem \ref{cnknots}, based on the theory of claspers, is the fact 
that the set of $C_k$-equivalence classes of knots forms an abelian group for all $k\ge 1$.   
As a matter of fact, his techniques apply to the string link case, and the fact that the group  $\mathcal{SL}(n)/C_k$ 
is abelian implies that the conjecture holds at the corresponding degree 
(see Appendix \ref{app} for further explanation).  
Hence Theorem~\ref{abeliansl}~(1) gives us the following.  
\begin{corollary}\label{cor1}
For any integer $k\leq 6$, two 2-string links cannot be distinguished by any finite type invariant of order less than $k$ if and only if they 
are $C_k$-equivalent. 
\end{corollary}
\begin{remark}
The Goussarov-Habiro conjecture for string links was shown to be true at low degree by various authors. 
It is easy to check for $k=2$, using the linking number \cite{MN}, and was 
proved for $k=3$ by the first author in \cite{jbjktr}.  
Massuyeau gave a proof for $k=4$, using algebraic arguments \cite{massuyeau}.  
In \cite{MYftisl}, the authors classified string links up to $C_k$-equivalence for $k\le 5$, 
by explicitly giving a complete set of low degree finite type invariants, 
and proved the Goussarov-Habiro Conjecture for $k\le 5$ as a byproduct.  
\end{remark}

Since we have two equivalence relations, concordance and $C_k$-equivalence, 
that provide group structures on the set of string links, it is natural to combine them to get a new group structure. 
We call this equivalence relation on string links generated by $C_k$-moves and concordance 
the {\em $C_k$-concordance} \cite{MYftisl}. 
The $C_k$-concordance is very closely related to Whitney tower concordance of order $k-1$ studied in \cite{CST} 
by Conant, Schneiderman and Teichner,\footnote{
These two equivalence relations are actually equivalent, as announced in \cite{CST}. }
and to finite type concordance invariants. 

In Section \ref{sec:ckconc}, we investigate whether the group  
$\mathcal{SL}(n)/(C_k+c)$ of $C_k$-concordance classes of $n$-string links is abelian. 
In the knot case, it is known that $\mathcal{SL}(1)/(C_k+c)$ is trivial for $k=1$ or $2$, and 
that $\mathcal{SL}(1)/(C_k+c)$ is isomorphic to ${\Bbb Z}/2{\Bbb Z}$ for $k\geq 3$ \cite{ng,MYftisl}.
We show the following. 
\begin{theorem}\label{abelian2}
\begin{enumerate}
 \item The group $\mathcal{SL}(2)/(C_k+c)$ is abelian if and only if $k\leq 8$. 
 \item For $n\geq 3$, the group $\mathcal{SL}(n)/(C_k+c)$ is abelian if and only if $k\leq 2$. 
\end{enumerate}
\end{theorem}
\noindent The proof of (1) is given in Sections \ref{sec:abelianc} and \ref{sec:ruby}. 
In Section \ref{sec:proof0}, we prove the `only if' part in (2).
The `if' part is actually easy to see. 
Indeed, as mentioned above, $\mathcal{SL}(n)/C_1$ is a trivial group and 
$\mathcal{SL}(n)/C_2$ is an abelian group classified by the linking number. 
Since the latter is a $C_2$-concordance invariant, we have that $\mathcal{SL}(n)/C_2$ and $\mathcal{SL}(n)/(C_2+c)$ are isomorphic. 

Since two $C_k$-concordant string links share all finite type concordance invariants of degree less than $k$, 
and it is natural to ask, parallel to the Goussarov-Habiro conjecture, whether the converse is also true. 
This question was raised by the authors in \cite{MYftisl}.
As in the case of the $C_k$-equivalence, Habiro's arguments apply if $\mathcal{SL}(n)/(C_k+c)$ is abelian. 
Hence we have the following corollary. 
\begin{corollary}\label{cor2}
For any integer $k\leq 8$, two 2-string links cannot be distinguished by any finite type 
concordance invariant of order less than $k$ if and only if they 
are $C_k$-concordant. 
\end{corollary}

\begin{acknowledgments}
The authors are grateful to Tetsuji Kuboyama for writing the computer program for Milnor invariants used in Section \ref{sec:ruby}, and to Kazuo Habiro for discussions that led to Lemmas \ref{lem:commute1} and \ref{lem:commute2}. 
The first author is supported by the French ANR research project ANR-11-JS01-00201. 
The second author is supported by a JSPS Grant-in-Aid for Scientific Research (C) ($\#$23540074). 
\end{acknowledgments}
\section{Claspers} \label{claspers}

We recall here the main definitions and properties of the theory of claspers, which is one of the main tools of this paper.
For convenience, we restrict ourselves to the case of string links.
For a general definition of claspers, we refer the reader to \cite{H}.

\subsection{A brief review of clasper theory} \label{review}

Let $L$ be a string link.
An embedded surface $g$ is called a {\em graph clasper} for $L$ if it satisfies the following three conditions:
\be
\item $g$ is decomposed into disks and bands, called {\em edges}, each of which connects two distinct disks.
\item The disks have either 1 or 3 incident edges, and are called {\em leaves} or {\em nodes} respectively.
\item $g$ intersects $L$ transversely, and the intersections are contained in the union of the interiors of the leaves.
\ee
In particular, if a graph clasper is a simply connected, we call it a \emph{tree clasper}.  
A graph clasper for a string link $L$ is \emph{simple} if each of its leaves intersects $L$ at one point.

The degree of a connected graph clasper $g$ is defined as half the number of nodes and leaves.
We call a degree $k$ connected graph clasper a \emph{$C_k$-graph}.
A tree clasper of degree $k$ is called a \emph{$C_k$-tree}. 

\begin{convention}\label{conv_star}
Throughout this paper, we make use of the drawing convention for claspers of \cite[Fig. 7]{H}, except for the following:
a $\oplus$ (resp. $\ominus$) on an edge represents a positive (resp. negative) half-twist.
(This replaces the convention of a circled $S$ (resp. $S^{-1}$) used in \cite{H}.)   
\end{convention}

Given a graph clasper $g$ for a string link $L$, there is a procedure to construct,
in a regular neighbourhood of $g$, a framed link $\gamma(g)$.
There is thus a notion of \emph{surgery along $g$}, which is defined as surgery along $\gamma(g)$.
There exists a canonical diffeomorphism between $D^2\times I$ and the manifold $(D^2\times I)_{\gamma(g)}$
fixing the boundary, and surgery along the $C_k$-graph $g$ can thus be regarded as an operation on $L$ in the (fixed)
ambient space  $D^2\times I$.
We say that the resulting string link $L_g$ in $D^2\times I$ is obtained from $L$ by surgery along $g$.
In particular, surgery along a simple $C_k$-tree is a local move as illustrated in Figure~\ref{ckmove}. 

Throughout this paper, we will often define string links in terms of claspers for the trivial string link,
implicitely referring to the result of surgery along this clasper.

\begin{figure}[!h]
 \begin{center}
 \includegraphics{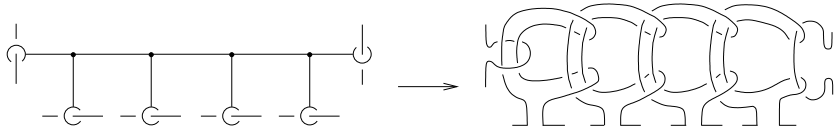}
 \caption{Surgery along a simple $C_5$-tree.} \label{ckmove}
 \end{center}
\end{figure}

The $C_k$-equivalence (as defined in the introduction) coincides with the equivalence relation on string links
generated by surgeries along $C_k$-graphs and isotopies.
In particular, it is known that two links are $C_k$-equivalent if and only if they are related by surgery along
simple $C_k$-trees \cite[Thm. 3.17]{H}.

This family of equivalence relations becomes finer as the index increases, that is, 
the $C_k$-equivalence implies the $C_m$-equivalence if $k>m$.

A string link is called {\em $C_k$-trivial} if it is $C_k$-equivalent to the trivial string link.

\subsection{Standard calculus of Claspers}

In this subsection, we summarize several properties of claspers.
Although similar statements hold in a more general context, it will be convenient to state these results 
for the trivial string link $\mathbf{1}$.
Proofs are omitted, since they involve the same techniques as in \cite[\S 4]{H}, where similar statements appear.

\begin{lemma} \label{lem:slide}
  Let $t_1\cup t_2$ be a disjoint union of a $C_{k_1}$-graph and a $C_{k_2}$-graph for $\mathbf{1}$.
  Let $t'_1\cup t'_2$ be obtained from $t_1\cup t_2$ by sliding a leaf of $t_1$ accross a leaf of $t_2$,
  see Figure \ref{fig:slide}.
  \begin{figure}[!h]
  \includegraphics{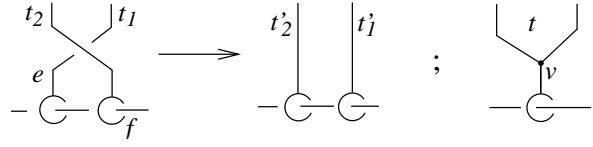}
  \caption{Sliding a leaf.    }\label{fig:slide}
 \end{figure}
  Then
  $$ \mathbf{1}_{t_1\cup t_2} \stackrel{C_{k_1+k_2+1}}{\sim} \mathbf{1}_{t'_1\cup t'_2}\cdot \mathbf{1}_t, $$
  where $t$ is a $C_{k_1+k_2}$-graph obtained from $t_1\cup t_2$ by inserting a vertex $v$
  in the edge $e$ of $t$ and connecting $v$ to the edge incident to $f$
  as shown in Figure \ref{fig:slide}.
\end{lemma}

\begin{lemma} \label{lem:cc}
  Let $t_1\cup t_2$ be a disjoint union of a $C_{k_1}$-graph and a $C_{k_2}$-graph for $\mathbf{1}$.
  Let $t'_1\cup t'_2$ be obtained from $t_1\cup t_2$ by changing a crossing of an edge of $t_1$ with
  an edge of $t_2$.
  Then
  $$ \mathbf{1}_{t_1\cup t_2} \stackrel{C_{k_1+k_2+2}}{\sim} \mathbf{1}_{t'_1\cup t'_2}\cdot \mathbf{1}_t, $$
  where $t$ is a $C_{k_1+k_2+1}$-graph.
\end{lemma}

\begin{remark}\label{rem:commute}
  By combining the two previous lemmas, we have the following.
  Let $L_1$ (resp. $L_2$) be a $C_{k_1}$-trivial (resp. $C_{k_2}$-trivial) $n$-string link, for some $n\ge 1$.
  Then $L_1\cdot L_2$ is $C_{k_1+k_2}$-equivalent to $L_2\cdot L_1$.
\end{remark}

\begin{lemma}\label{lem:split}
Let $g$ be a $C_k$-graph for $\1_n$. Let $f_1$ and $f_2$ be two disks obtained by splitting a leaf $f$ of $g$ 
along an arc $\alpha$ as shown in Figure \ref{split}.
 \begin{figure}[!h]
  \includegraphics{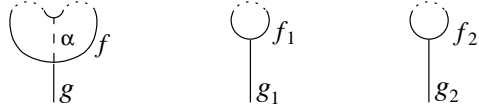}
  \caption{The $3$ claspers are identical outside a small ball, where they are as depicted. }\label{split}
 \end{figure}
 Then, 
 $$ (\1_n)_g \stackrel{C_{k+1}}{\sim} (\1_n)_{g_1}\cdot (\1_n)_{g_2}, $$ 
 where $g_i$ denotes the $C_k$-graph for $\1_n$ obtained from $g$ by replacing $f$ by $f_i$ ($i=1,2$), see Figure \ref{split}.  
\end{lemma}

\begin{lemma}\label{lem:twist}
  Let $t$ be a $C_k$-graph for $\mathbf{1}$, and let $t'$ be a $C_k$-graph obtained from $t$
  by adding a half-twist on an edge. Then
  $\mathbf{1}_t\cdot \mathbf{1}_{t'} \stackrel{C_{k+1}}{\sim} \mathbf{1}$.
\end{lemma}

Claspers also satisfy relations analogous to the AS, IHX and STU relations for Jacobi diagrams \cite{BNv}.
\begin{lemma}\label{asihxstu}

(AS). Let $t$ and $t'$ be two $C_k$-graphs for $\mathbf{1}$ which differ only in a small ball
  as depicted in Figure \ref{asihxstu_fig}. Then
  $\mathbf{1}_{t}\cdot \mathbf{1}_{t'} \stackrel{C_{k+1}}{\sim}\mathbf{1}$.
  \begin{figure}[!h]
   \includegraphics{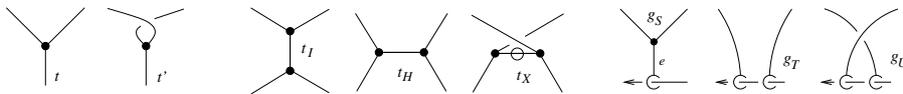}
   \caption{The AS, IHX and STU relations.  }\label{asihxstu_fig}
  \end{figure}

(IHX). Let $t_I$, $t_H$ and $t_X$ be three $C_k$-graphs for $\mathbf{1}$ which differ only in a small ball
  as depicted in Figure \ref{asihxstu_fig}. Then
  $\mathbf{1}_{t_I} \stackrel{C_{k+1}}{\sim} \mathbf{1}_{t_H}\cdot \mathbf{1}_{t_X}$.

(STU). Let $g_S$, $g_T$ and $g_U$ be three $C_k$-graphs for $\mathbf{1}$ which differ only in a small ball
  as depicted in Figure \ref{asihxstu_fig}. Then
  $\mathbf{1}_{g_S}\cdot \mathbf{1}_{g_T} \stackrel{C_{k+1}}{\sim} \mathbf{1}_{g_U}$.
\end{lemma}

In the rest of the paper, we will simply refer to Lemma \ref{asihxstu} as the AS, IHX and STU relations.

By the IHX and STU relations, one can easily check the following.
\begin{lemma}\label{fork}
Let $t$ be a $C_k$-graph for $\mathbf{1}$.

(1).~Suppose that there exists a $3$-ball which intersects $t$ as on the left-hand side of
  Figure \ref{fig:fork}.
  Then
  $\mathbf{1}_t\stackrel{C_{k+1}}{\sim} \mathbf{1}_c$,
  where $c$ is a $C_k$-graph as shown in the figure. 
  \begin{figure}[!h]
   \includegraphics{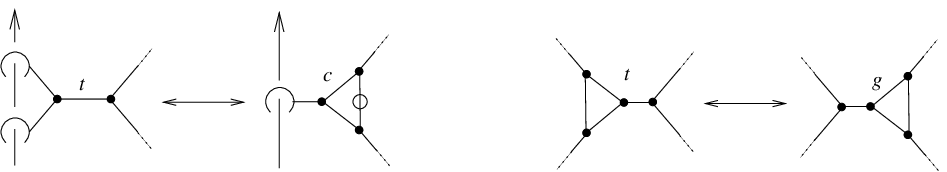}
   \caption{}\label{fig:fork}
  \end{figure}

(2).~Suppose that there exists a $3$-ball which intersects $t$ as on the right-hand side of
  Figure \ref{fig:fork}.
  Then
  $\mathbf{1}_t\stackrel{C_{k+1}}{\sim} \mathbf{1}_g$,
  where $g$ is a $C_k$-graph as shown in the figure. 
\end{lemma}

\subsection{Commutativity lemmas}
In this section, we collect a couple extra technical, purely algebraic lemmas.

\begin{lemma}\label{lem:commute1}
Let $l$ and $k$ be positive integers $(l<k)$.
Let $G$ be a $C_l$-trivial string link and let $G'$ be the inverse of $G$ in $\mathcal{SL}(n)/C_{l+1}$.
If $G$ commutes with any $C_{l+1}$-trivial string link in $\mathcal{SL}(n)/C_{k}$, 
then $G$ and $G'$ commute in $\mathcal{SL}(n)/C_{k}$.
\end{lemma}

\begin{proof}
Since $G\cdot G'$ is $C_{l+1}$-trivial, $G$ commutes with $G\cdot G'$ in
$\mathcal{SL}(n)/C_{k}$. Hence
\[G\cdot G' \stackrel{C_{k}}{\sim}
G^{-1}\cdot G\cdot (G\cdot G') \stackrel{C_{k}}{\sim} G^{-1}\cdot (G\cdot G')\cdot G
\stackrel{C_{k}}{\sim} G'\cdot G,\]
where $G^{-1}$ denotes the inverse of $G$ in $\mathcal{SL}(n)/C_{k}$.
\end{proof}

Although seemingly very technical, the next lemma turns out to be rather natural 
in the proofs of our main results.
\begin{lemma}\label{lem:commute2}
Let $L$ be an $n$-string links satisfying the following three conditions, 
for some integers $p$, $l$, $m$ and $k$ ($p<k$, $l<m<k$) : 
\begin{itemize}
 \item $L^2$ is $C_p$-equivalent to a central element in $\mathcal{SL}(n)/C_k$,
 \item $L$ commutes with any $C_l$-trivial string link in $\mathcal{SL}(n)/C_m$,
 \item $L$ commutes with any $C_m$-trivial string link in $\mathcal{SL}(n)/C_k$. 
\end{itemize}
Suppose, moreover, that $\mathcal{SL}(n)/C_{k}$ has no $2$-torsion, and that 
$C_l$-trivial and $C_p$-trivial string links all commute in $\mathcal{SL}(n)/C_k$.
Then $L$ commutes with any $C_l$-trivial string link in $\mathcal{SL}(n)/C_{k}$.
\end{lemma}

\begin{remark}\label{rem:trivial}
Observe that a string link commutes with any $C_l$-trivial string link in $\mathcal{SL}(n)/C_k$ ($l<k$) 
if and only if it commutes with $\1_t$ for any $C_s$-tree $t$ with $l\leq s\leq k-1$.
Indeed, let $T$ and $T'$ be two $C_{l}$-equivalent string links. 
Since $T$ is $C_1$-trivial, it is not hard to see, using Lemmas~\ref{lem:slide} and \ref{lem:cc}, that $T'$ is $C_{k}$-equivalent to
a product $T\cdot T_{l}\cdot T_{l+1}\cdots T_{k-1}$
where $T_i$ is a product of string links, each obtained from $\1$ by surgery along a single $C_i$-tree $(l\le i\le k-1)$.
In particular, each $T_i$ is a $C_i$-trivial string link, and $T_{l}\cdot T_{l+1}\cdots T_{k-1}$ is thus a $C_l$-trivial string link.
\end{remark}

\begin{proof}
Let $G$ be a $C_l$-trivial string link. Let us show that $L$ and $G$ commute in
$\mathcal{SL}(n)/C_{k}$.
Since $L^2$ is $C_p$-equivalent to a central element $C$, we have that
$L^2$ is $C_k$-equivalent to $C \cdot H$ for some $C_p$-trivial string link $H$
(by an argument similar to that in Remark \ref{rem:trivial} above).
Thus we have
\[G\cdot L^2\stackrel{C_{k}}{\sim} G\cdot C\cdot H
\stackrel{C_{k}}{\sim} C\cdot H\cdot G
\stackrel{C_{k}}{\sim} L^2\cdot G, \]
where the second equivalence follows from the assumption 
that $C_l$-trivial and $C_p$-trivial string links commute in $\mathcal{SL}(n)/C_k$).  
Hence
\[\1_n \stackrel{C_{k}}{\sim} [L^2,G]
\stackrel{C_{k}}{\sim} L[L,G]L^{-1}[L,G],\]
where $[x,y]$ denotes the commutator $xyx^{-1}y^{-1}$ of $x$ and $y$.
Since $L$ and $G$ commute in $\mathcal{SL}(n)/C_{m}$ (by assumption on $L$), 
the commutator $[L,G]$ is $C_{m}$-trivial.
Thus $L$ and $[L,G]$ commute in $\mathcal{SL}(n)/C_{k}$, and  we have
$\1_n \stackrel{C_{k}}{\sim} [L,G]^2$.
Since $\mathcal{SL}_{k-1}(n)/C_{k}$ has no 2-torsion,
we have that $[L,G]$ is $C_k$-trivial, and hence $L$ and $G$ commute in
$\mathcal{SL}(n)/C_{k}$.
\end{proof}
\section{The group of $C_k$-equivalence classes of $2$-string links. } \label{sec:proof}
In this section, we prove Theorems \ref{abeliansl} and \ref{prop:c7}.
\subsection{Abelian cases: $k\le 6$}\label{sec:abelian}
Since the $C_k$-equivalence implies the $C_m$-equivalence if $k>m$, it is sufficient to show that $\mathcal{SL}(2)/C_6$ is abelian.  
This is done in Proposition~\ref{lem:abelian6}, by first providing a set of generators for this group 
and then by showing that any two of these generators commute in $\mathcal{SL}(2)/C_6$.  
(Note that it was shown by Habiro that $\mathcal{SL}(n)/C_k$ is a finitely generated group for any $n$ and $k$ \cite{H}.)

Let the $2$-string links $I$, $Y$, $Y'$, $H$, $X$, $D$, $S^1_\alpha$ and $S^2_\alpha$ ($\alpha\in S_3$)
obtained from $\mathbf{1}_2$ by surgery along the tree clasper $i$, $y$, $y'$, $h$, $x$, $d$, 
$s^1_\alpha$ and $s^2_\alpha$, represented in Figure \ref{fig:tree}, respectively.   
 \begin{figure}[!h]
  \includegraphics{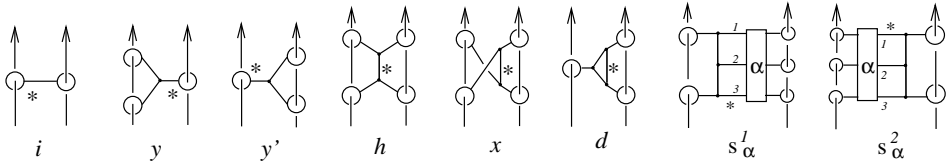}
  \caption{Here, when representing a clasper $c$ with an edge marked by a $\ast$, we implicitly also define the clasper
$c^{\ast}$ obtained from $c$ by inserting a \emph{negative half twist} in the $\ast$-marked edge.
Thus, when introducing the string link $C$ obtained from $\mathbf{1}$ by surgery along $c$, 
we implicitly also introduce the string link $C^\ast$ obtained by surgery along $c^\ast$. }
  \label{fig:tree}
 \end{figure} 

Let $\mathcal{SL}_i(2)$ denote the set of $C_i$-trivial $2$-string links ($i\ge 1$).
It is shown in \cite[\S 4.1 to 4.3.2]{MYftisl} that, for $i\in\{1,2,3,4\}$, the group $\mathcal{SL}_i(2)/C_{i+1}$ has generating 
set $\mathcal{H}_i = \mathcal{H}^m_i\cup \mathcal{H}^l_i$, where\footnote{
In \cite{MYftisl}, the generating set $\mathcal{H}^m_3 = \{H,X\}$ is used instead. 
Here, it will be convenient to use the element $D$ rather than $X$ as generator. 
This is possible since, by the STU relation, we have $D\cdot H\stackrel{C_4}{\sim} X$.}  
$$ \mathcal{H}^m_1 = \{I\}, \quad \mathcal{H}^m_2 = \{Y\}, \quad  \mathcal{H}^m_3 = \{H,D\},  
\quad \mathcal{H}^m_4 = \{\si^1_{id},\si^1_{(12)},\si^1_{(123)},\si^1_{(13)},\si^2_{id}\}, $$
and where $\mathcal{H}^l_i$ is a set of \emph{local} generators. 
Here, we say that an element of $\mathcal{SL}(2)$ is local if there exists a 3-ball whose boundary 
intersects it at only two points, such that an homotopy of this ball to a point produces the 
trivial $2$-string link.
(In other words, local elements consists of a local knot on one strand).  
Clearly local elements are central in $\mathcal{SL}(2)$.)  

Let us also fix a generating set $\mathcal{H}_5$ for $\mathcal{SL}_5(2)/C_{6}$. 

For any $L\in \mathcal{H}_{i}$ ($i\le 5$), we fix an inverse $L^\ast$ in $\mathcal{SL}_i(2)/C_{i+1}$. 
In particular, we use the convention of Figure \ref{fig:tree} 
to pick inverses for elements of $\mathcal{H}_{i}$ ($i\le 4$).  
We call elements of $\mathcal{H}_i$ and their inverses \emph{generators of degree $i$}.

Before stating Proposition~\ref{lem:abelian6}, 
we make a few simple, yet useful observations. 
First, notice that 
\begin{lemma}\label{lem:I}
  Both $I$ and $I^\ast$ are central in $\mathcal{SL}(2)$. 
\end{lemma}
\begin{proof}
This is shown in Figure \ref{fig:rotate} for the element $I$ (the case of $I^\ast$ is similar).           
  \begin{figure}[!h]
  \includegraphics{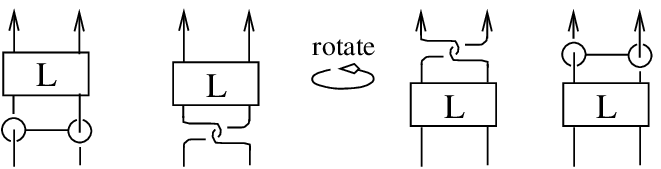}
  \caption{}\label{fig:rotate}
 \end{figure}
Let $L\in \mathcal{SL}(2)$. Then $I\cdot L$ is as represented on the left-hand side of the figure. 
Rotating the two strings by $360$ degrees about the vertical axis (fixing the endpoints) yields an isotopic string link, 
which is precisely $L\cdot I$. 
\end{proof}
Next, we have
\begin{fact}\label{eq:Y}
 $Y$ and $Y^*$ are ambient isotopic to $Y'$ and $(Y')^*$ respectively. 
\end{fact}
This is the string link version of the usual symmetry property of the Whitehead link, 
as illustrated in Figure \ref{fig:whitehead} below.
 \begin{figure}[!h]
  \includegraphics{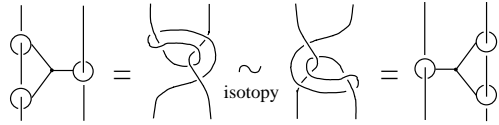}  
  \caption{Symetry property of the Whitehead string link.}    
  \label{fig:whitehead}
 \end{figure}

In the following, however, we will use the term `symmetric' for another property. 
Given an $n$-string link $L$, we denote by $\overline{L}$ its image under orientation-reversal of all its strands. 
The string link $L$ is called \emph{symmetric} if it is isotopic to $\overline{L}$. 
For example, the string links $I$, $I^*$, $Y$, $Y^*$, $Y'$, $Y'^*$, $H$, $H^*$, $D$ and $D^*$ are all symmetric. 
Notice that this is nicely reflected in the symmetry of the clasper defining these string links. 

We can now prove the first part of Theorem \ref{abeliansl}.  

\begin{proposition}\label{lem:abelian6}
The group $\mathcal{SL}(2)/C_6$ is abelian.  
\end{proposition}

\begin{proof}
We use the generating sets $\mathcal{H}_i$ for $\mathcal{SL}_i(2)/C_{i+1}$ specified above ($1\le i\le 5$). 
In order to prove Proposition~\ref{lem:abelian6}, 
it suffices to show that, for any two elements 
$A$ and $B$ in $\cup_{i=1}^{5} \mathcal{H}_{i}$, we have that $A$, $A^*$, $B$ and $B^*$ 
commute with each other in $\mathcal{SL}(2)/C_6$. 

By Remark \ref{rem:commute}, two generators of degree $k$ and $k'$ commute in $\mathcal{SL}(2)/C_6$ if $k+k'\ge 6$. 
Hence all generators of degree $5$ are central in $\mathcal{SL}(2)/C_6$. 
Remark \ref{rem:commute} also implies that $H$, $H^\ast$, $D$ and $D^\ast$ commute with each other, 
Moreover, we may safely ignore local generators, since these are central elements in $\mathcal{SL}(2)$, 
as well as $I$ and $I^\ast$ (by Lemma \ref{lem:I}). 

So we only need to check that the generators of degree $2$,  
$Y$ and $Y^*$, both commute with $H$, $H^*$, $D$ and $D^*$, and that $Y$ and $Y^*$ commute with each other.  

Let us first show that $Y$ and $H$ commute in $\mathcal{SL}(2)/C_6$. 
The proof is given in Figures \ref{fig:YH1} and \ref{fig:YH2} as follows.
Consider the product $H\cdot Y$. 
By Lemmas \ref{lem:slide} and \ref{lem:twist}, we have that $H\cdot Y\stackrel{C_6}{\sim} G\cdot K$, 
where $G$ and $K$ are shown in Figure \ref{fig:YH1}.  
  \begin{figure}[!h]
  \includegraphics{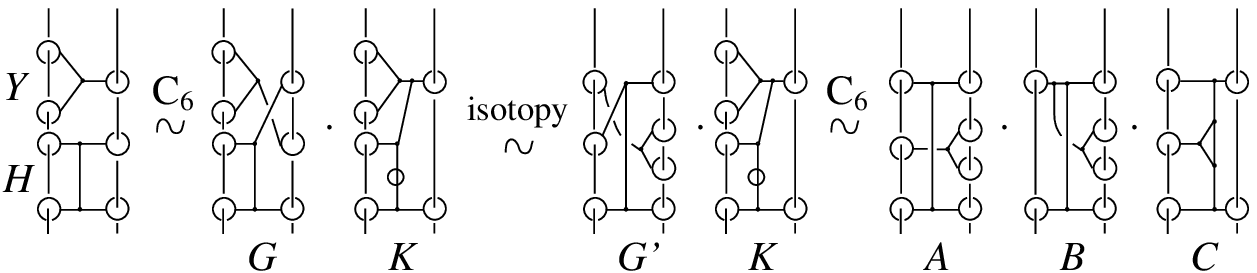}
  \caption{}\label{fig:YH1}
 \end{figure}
Observe that $G$ `locally' contains a copy of $Y$, so that we can use Fact~\ref{eq:Y} to show that $G$ is 
isotopic to the string link $G'$ represented in Figure \ref{fig:YH1}.
Now, by a second application of Lemma \ref{lem:slide}, $G'$ is $C_6$-equivalent to the product 
of the two string links $A$ and $B$ represented on the right-hand side of the figure.  
Also, Lemmas \ref{lem:twist} and \ref{fork} imply that $K\stackrel{C_6}{\sim} C$, 
where $C$ is also given in Figure \ref{fig:YH1}.   

Let us now focus on the string link $B$.
By the IHX relation and Lemma \ref{lem:twist}, we have 
$B\stackrel{C_6}{\sim} B'\cdot B''$, as illustrated in Figure \ref{fig:YH2}.
  \begin{figure}[!h]
  \includegraphics{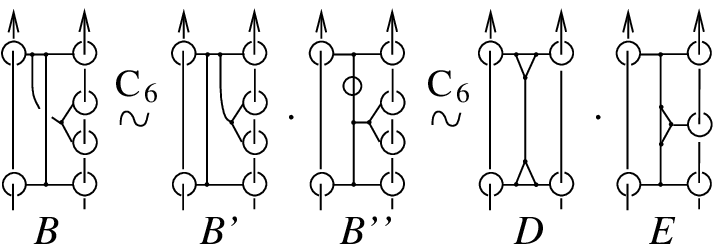}
  \caption{}\label{fig:YH2}
 \end{figure}
Applying Lemmas \ref{fork} and \ref{lem:twist} to the string link $B'$ and $B''$, 
we deduce that $B\stackrel{C_6}{\sim} D\cdot E$, where $D$ and $E$ are represented on the right-hand side 
of Figure \ref{fig:YH2}.

Hence we have shown that $H\cdot Y\stackrel{C_6}{\sim}  A\cdot C\cdot D\cdot E$. 
Observe that $A\stackrel{C_6}{\sim}\overline{A}$ by Lemma~\ref{lem:cc},  
and that the three string links $C$, $D$ and $E$ are symmetric. 
We thus have 
 $$ H\cdot Y\stackrel{C_6}{\sim} \overline{H\cdot Y} =  Y\cdot H, $$
where the equality follows from the fact that $Y$ and $H$ are also symmetric. 

The same argument, together with Lemma \ref{lem:twist}, shows that $Y$ commutes with $H^\ast$ 
and that $Y^\ast$ commutes with $H$ and $H^\ast$.  

Similarly, we now show that $Y$ and $Y^\ast$ commute with $D$ and $D^\ast$ as follows. 
Starting with the product $D\cdot Y$ and applying Lemmas \ref{lem:slide} and \ref{lem:twist} twice as shown in  
Figure \ref{fig:YD1}, we obtain that $D\cdot Y\stackrel{C_6}{\sim} M\cdot N\cdot O$, 
where $M$, $N$ and $O$ are represented on the right-hand side of the figure. 
  \begin{figure}[!h]
  \includegraphics{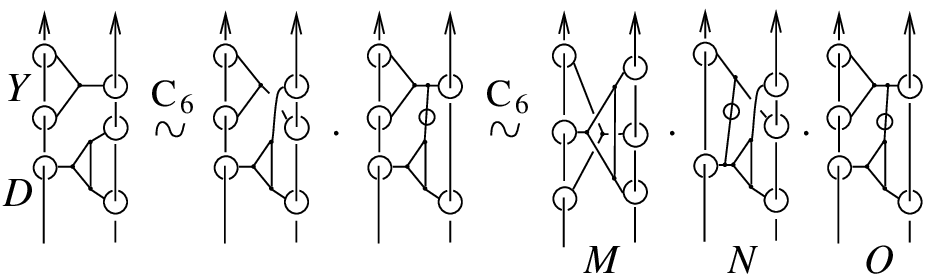}
  \caption{}\label{fig:YD1}
 \end{figure}
The STU relation shows that $N\stackrel{C_6}{\sim} P\cdot Q$, 
where $P$ and $Q$ are represented in Figure \ref{fig:YH2}.  
Now, by Lemmas \ref{lem:twist} and \ref{fork} we see that $O$ and $P$ are 
both $C_6$-equivalent to the string link $R$ shown in Figure \ref{fig:YH2}.
  \begin{figure}[!h]
  \includegraphics{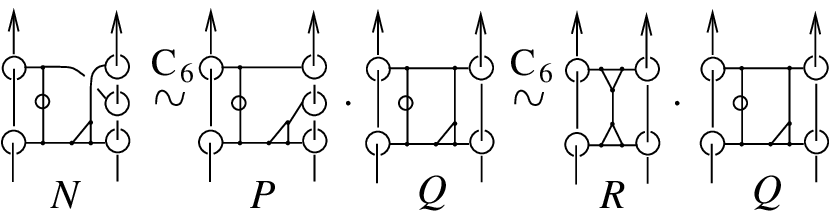}
  \caption{}\label{fig:YD2}
 \end{figure}
Summarizing, we have shown that 
$$  D\cdot Y\stackrel{C_6}{\sim} M\cdot R^2\cdot Q. $$
Now, we have that $M\stackrel{C_6}{\sim} \overline{M}$ by Lemma \ref{lem:cc},  
and we also have $Q\stackrel{C_6}{\sim} Q'$ by Lemma \ref{fork}(2). Hence we obtain that 
$D\cdot Y\stackrel{C_6}{\sim} \overline{D\cdot Y} = Y\cdot D$,
where the equality follows from the fact that $Y$ and $D$ are both symmetric. 

The same argument shows that $Y$ commutes with $D^\ast$ 
and that $Y^\ast$ commutes with $D$ and $D^\ast$.

It only remains to prove that $Y$ and $Y^\ast$ commute. 
So far, we proved that $Y$ commutes with any $C_3$-trivial string link in $\mathcal{SL}(2)/C_6$. 
Since by Lemmas \ref{lem:twist}, $Y^\ast$ is the inverse of $Y$ in $\mathcal{SL}(2)/C_3$, we can apply  
Lemma \ref{lem:commute1} to conclude that $Y$ and $Y^\ast$ commute in $\mathcal{SL}(2)/C_6$. 
\end{proof}

\subsection{Non-abelian case: $k\ge 8$}\label{sec:proof2}

Denote by $\mathcal{A}(2)$ the $\mathbb{Q}$-vector space of 
Jacobi diagrams on two strands, modulo the Framing Independance (FI) and STU relation,  
and denote by $\mathcal{A}_k(2)$ the subspace generated by degree $k$ elements 
(see e.g. \cite{BNv} for the definitions). 
The stacking product $\cdot$ endows $\mathcal{A}(2)$ with an algebra structure. 
As is well-known, $\mathcal{A}(2)$ coincides with the space of chord diagrams on two strands modulo the FI and 4T relations \cite{BNv}. 

In \cite{DK}, Duzhin and Karev showed that $\mathcal{A}(2) / \mathcal{A}_8(2)$, hence  $\mathcal{A}(2)$,  
is noncommutative. 
More precisely, let $D_H\in \mathcal{A}_3(2)$ and $D_S\in \mathcal{A}_4(2)$ be the two Jacobi diagrams shown in Figure \ref{fig:diags}.
   \begin{figure}[!h]
   \includegraphics{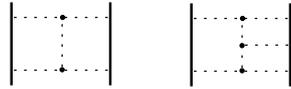}
   \caption{The Jacobi diagrams $D_H\in \mathcal{A}_3(2)$ and $D_S\in \mathcal{A}_4(2)$.}\label{fig:diags}
  \end{figure}
Then we have that $D_H\cdot D_S\neq D_S\cdot D_H$ in $\mathcal{A}_7(2)$ \cite[Prop. 1]{DK}. 
We now show how this result implies that $\mathcal{SL}(2)/C_8$ is not abelian.

For $k\ge 0 $, let $J_k(2)$ denote the subgroup of $\mathbb{Z}\mathcal{SL}(2)$ generated by singular $2$-string links with $k$ double points, 
via the Vassiliev skein relation 
$$ \double = \Over - \under. $$
By definition, the difference of two string links is in $J_{k+1}(2)$ if and only if 
they cannot be distinguished by any finite type invariant of degree $\le k$. 
There is a well-known isomorphism 
$$ \xi_k : \mathcal{A}_k(2)\rightarrow (J_k(2) / J_{k+1}(2))\otimes \mathbb{Q} $$ 
which ``maps chords to double points'', with inverse given by the Kontsevich integral.  

Now recall from Section \ref{sec:abelian} that $H$ and $S^1_{Id}$ denote the $2$-string links 
obtained from $\mathbf{1}_2$ by surgery along the tree clasper $h$ and $s^1_{Id}$ represented in Figure \ref{fig:tree}, respectively.  
(For simplicity, we will use here the simpler notation $S$ for the string link $S^1_{Id}$.)  
Then the image of $D_H\cdot D_S - D_S\cdot D_H$ by $\xi_7$ is (up to a sign) the difference 
$H\cdot S - S\cdot H$  
(this can be checked using a standard argument on the good behavior of the Kontsevich integral on alternate sums defined by claspers, 
see e.g. \cite[Appendix E]{ohtsuki}). 
This shows that $H\cdot S$ and $S\cdot H$ can be distinguished by some degree $7$ finite type invariant. 
Since two $C_k$-equivalent (string) links cannot be distinguished by any finite type invariant of degree $<k$, 
we deduce that $H\cdot S$ and $S\cdot H$ are not $C_8$-equivalent.

\subsection{The case $k=7$}

In this section, we prove Theorem \ref{prop:c7}, hence we suppose that $\mathcal{SL}(2) / C_7$ has no $2$-torsion. 

We proceed as in Section \ref{sec:abelian}, 
using the generating sets $\mathcal{H}_i$ for $\mathcal{SL}_i(2) / C_{i+1}$ ($i\le 5$). 
By Remark \ref{rem:commute}, elements of $\mathcal{SL}_6(2) / C_7$ are central in $\mathcal{SL}(2)/C_7$, 
so we only need to consider these generators of degree $\le 5$. 
As above, we may also safely ignore local generators, as well as $I$ and $I^\ast$. 
Actually, by Lemma \ref{lem:I} and Remark~\ref{rem:commute}, we only have to check the following commutativity properties : 
\begin{enumerate}
 \item[(3,3)] Generators of degree $3$ commute with each other. 
 \item[(2,4)] Generators of degree $4$ commute with generators of degree $2$.
 \item[(2,3)] Generators of degree $3$ commute with generators of degree $2$.
 \item[(2,2)] Generators of degree $2$ commute with each other. 
\end{enumerate}
\noindent (Recall from Section \ref{sec:abelian} that by generator of degree $k$ we mean any element of $\mathcal{H}_k$ or its fixed inverse.) 

{\bf Case $(3,3)$}: 
We have to show that $H,H^*, D$ and $D^*$ commute with each other in $\mathcal{SL}(2)/C_7$.  
Since, by Remark \ref{rem:commute}, each of these elements commute with any $C_4$-trivial string link in $\mathcal{SL}(2) / C_7$, 
we can apply Lemma \ref{lem:commute1} to show that $H$ and $H^\ast$ (resp. $D$ and $D^\ast$) commute  in $\mathcal{SL}(2) / C_7$.
Moreover, we have the following result, whose proof is postponed to the end of this section. 
\begin{claim}\label{lem:dcentral}
 The string links $D^2$ and $(D^*)^2$ are both $C_4$-equivalent to a central element in $\mathcal{SL}(2)$. 
\end{claim}
\noindent It follows from this fact and Remark \ref{rem:commute} 
that $D$ and $D^*$ both fullfill the assumptions of Lemma \ref{lem:commute2} with 
$(p,l,m,k)=(4,3,6,7)$. 
Applying the lemma then proves that $D$ (resp. $D^\ast$) commutes with any $C_3$-trivial string link in $\mathcal{SL}(2) / C_7$, 
and in particular with both $H$ and $H^\ast$.  

{\bf Case $(2,4)$}: 
We have
\begin{claim}\label{lem:ycentral}
 The string links $Y^2$ and $(Y^*)^2$ are both $C_3$-equivalent to a central element in $\mathcal{SL}(2)$. 
\end{claim}
\noindent (The proof of Claim~\ref{lem:ycentral} uses the exact same arguments as the proof of Claim~\ref{lem:dcentral}, 
although it turns out to be actually significantly simpler.)
 
\noindent Claim \ref{lem:ycentral} and Remark \ref{rem:commute} ensure that both $Y$ and $Y^\ast$ satisfy the assumptions of 
Lemma \ref{lem:commute2} with $(p,l,m,k)=(3,4,6,7)$. 
Applying the lemma then shows that $Y$ and $Y^\ast$ commute with any generator of degree $4$ in $\mathcal{SL}(2) / C_7$. 

{\bf Case $(2,3)$}: 
Using Case~$(3,3)$ and Claim \ref{lem:ycentral}, we can now apply 
Lemma \ref{lem:commute2} to the string links $Y$ and $Y^\ast$ with $(p,l,m,k)=(3,3,5,7)$.
Applying the lemma then shows that $Y$ and $Y^\ast$ commute with any generator of degree $3$ in $\mathcal{SL}(2) / C_7$. 

{\bf Case $(2,2)$}: 
Let us show that $Y$ and $Y^*$ commute with each other in $\mathcal{SL}(2)/C_7$.   
By Case~$(2,3)$, we have that $Y$ commutes with any $C_3$-trivial string link in $\mathcal{SL}(2) / C_7$. 
Since by Lemma \ref{lem:twist}, $Y^*$ is the inverse of $Y$ in $\mathcal{SL}(2) / C_3$, 
we can apply Lemma \ref{lem:commute1} to get the desired result. 

We conclude this section with the proof of Claim \ref{lem:dcentral}. 
\begin{proof}[Proof of Claim \ref{lem:dcentral}]
Let us only give the proof for the $2$-string link $D$ (the proof for $D^\ast$ is easily deduced). 
We first notice that, as a consequence of Lemma \ref{fork}, the string link $D$ (and likewise, $D^*$) has a symmetry property similar 
to that of the Whitehead string link $Y$ in $\mathcal{SL}(2) / C_4$ (see Figure \ref{fig:whitehead}) :  
   \begin{figure}[!h]
   \includegraphics{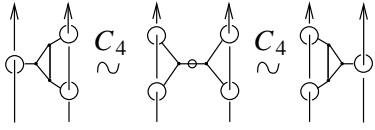}
   \caption{A Whitehead-type symmetry for the $2$-string link $D$.}\label{fig:symD}
  \end{figure}

Now consider the $C_3$-graph $g$ for $\1_2$ shown on the left-hand side of Figure \ref{fig:splitD}. 
Clearly, the $2$-string link $G$ obtained by surgery along $g$ is central in $\mathcal{SL}(2)$. 
By applying repeatedly Lemma \ref{lem:split}, we have that $G$ is $C_4$-equivalent to a product $\prod_{i=1}^8 (\1_2)_{g_i}$ 
of eight $2$-string links, each obtained from $\1_2$ by surgery along a simple $C_4$-graph, as shown in Figure \ref{fig:splitD}. 
By several isotopies, Lemma \ref{lem:twist} and the symmetry property of Figure \ref{fig:symD}, we deduce that 
$G$ is $C_4$-equivalent to the product $C\cdot (D^*)^2$, where $C=(\1_2)_{g_1}\cdot (\1_2)_{g_5}$ is clearly central in $\mathcal{SL}(2)$
(see Figure \ref{fig:splitD}). 
   \begin{figure}[!h]
   \includegraphics{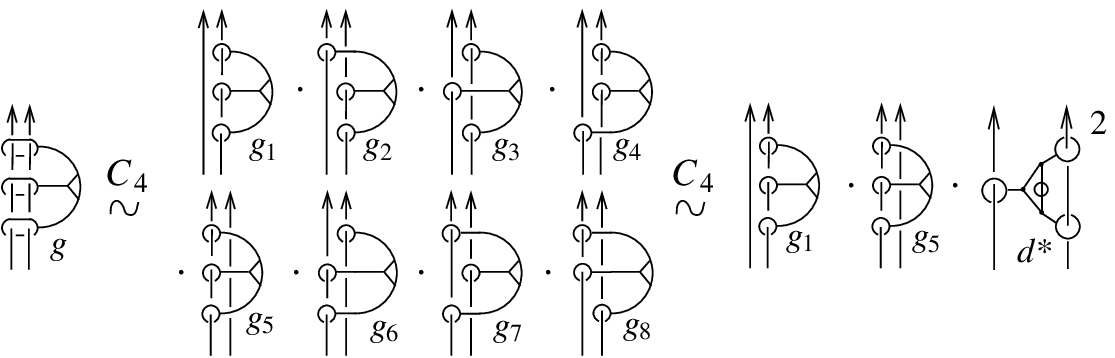}
   \caption{ }\label{fig:splitD}
  \end{figure}

Using Lemma \ref{lem:twist}, we thus have that $D^2\stackrel{C_4}{\sim} C\cdot G^{-1}$, 
where $G^{-1}$ denotes the inverse of the central element $G$ in $\mathcal{SL}(2)$. The result follows.  
\end{proof}
\begin{remark}\label{rem:by_hand}
The techniques used in Section \ref{sec:abelian} to prove that $\mathcal{SL}(2)/C_k$ is abelian for $k\le 6$ 
can also be applied to the case $k=7$ to some extent. 
More precisely, by fixing a generating set $\mathcal{H}_6$ for $\mathcal{SL}_6(2)/C_7$, 
one can try to show directly that, for any two elements 
$A$ and $B$ in $\cup_{i=1}^{6} \mathcal{H}_{i}$, we have that $A$, $A^*$, $B$ and $B^*$ 
commute with each other in $\mathcal{SL}(2)/C_7$.  
We can actually apply our methods, involving somewhat advanced calculus of claspers, to prove that this is indeed the case, 
except for \emph{one} computation that remains open. Namely, the case where $A=Y$ and $B=H$ remains open so far, 
and is the only missing case to establish the commutativity of $\mathcal{SL}(2)/C_7$ (without the $2$-torsion assumption).  
\end{remark}
\section{The group of $C_k$-concordance classes of string links. }\label{sec:ckconc}
In this section, we prove Theorem~\ref{abelian2} and Proposition \ref{nonabeliansl}. 
We start with a brief review on the $C_k$-concordance. 
\subsection{$C_k$-concordance}\label{sec:ckc}
Let $k$ and $n$ be positive integers.  
Two $n$-string links $L, L'$ are \emph{$C_k$-concordant} if there is a sequence 
$L=L_0,L_1, ... ,L_m=L'$ of $n$-string links such that for each $i\ge 1$, either $L_i\stackrel{C_{k+1}}{\sim} L_{i+1}$ 
or $L_i$ is concordant to $L_{i+1}$.  
We denote the $C_k$-concordance relation by $\stackrel{C_k+c}{\sim}$. 

A string link is  {\em $(C_k+c)$-trivial} if it is $C_k$-concordant to the trivial string link.

It is well-known that Milnor invariants are concordance invariants \cite{casson}.  
So by \cite[Thm. 7.1]{H}, Milnor invariants of length $\le k$ are $C_k$-concordance invariants.
Habegger and Masbaum showed that all \emph{rational} finite type concordance invariants of 
string links are given by Milnor invariants via the Kontsevich integral \cite{HMa}.  

It is known that surgery along graph claspers \emph{with loops} 
(i.e. graph claspers that are not tree claspers) implies concordance.  
\begin{lemma}(\cite{CT,GL})\label{graph} 
For any graph clasper with loop $g$ for $\1$, the string link $\1_g$ is concordant to $\1$.  
\end{lemma}

By combining this lemma and the STU relation, we have the following. 
\begin{lemma} \label{lem:STUc}
  Let $g_T$ and $g_U$ be two $C_k$-trees for $\mathbf{1}$ which differ only in a small ball
  as depicted in Figure \ref{asihxstu_fig}. Then
  $\mathbf{1}_{g_T} \stackrel{C_{k+1}+c}{\sim} \mathbf{1}_{g_U}$.
\end{lemma}

The simple algebraic argument used to prove Lemmas~\ref{lem:commute1} also hold for $C_k$-concordance.
That is, we have the following 
\begin{lemma}\label{lem:commute1'}
Let $l$ and $k$ be positive integers $(l<k)$. 
Let $G$ be a $(C_l+c)$-trivial string link  
and let $G'$ be the inverse of $G$ in 
$\mathcal{SL}(n)/(C_{l+1}+c)$. If $G$ commutes with any $(C_{l+1}+c)$-trivial string link 
in $\mathcal{SL}(2)/(C_{k}+c)$, 
then $G$ and $G'$ commute in $\mathcal{SL}(2)/(C_{k}+c)$.
\end{lemma}
\noindent Likewise, there is a `$C_k$-concordance version' of Lemma~\ref{lem:commute2} (see Lemma \ref{lem:commute3} below). 

\subsection{The ordered index}\label{sec:ordered}

In order to study $C_k$-concordance for string links, we use the notion of ordered index of a linear $C_k$-tree.  

Let $t$ be a simple tree clasper for a string link $L$. 
We call a leaf of $t$ an {\em$i$-leaf} if it intersects the $i$th component of $L$.  
The \emph{index} of $t$ is the collection of all integers $i$ such that $t$ contains an $i$-leaf, counted with multiplicities. 
For example, a simple $C_3$-tree of index $\{2,3^{(2)},5\}$ for $L$ intersects component $3$ twice and components $2$ and $5$ once 
(and is disjoint from all other components of $L$).  

For $k\ge 3$, a $C_k$-tree $G$ having the shape of the tree clasper in Figure~\ref{ckmove} 
is called a \emph{linear} $C_k$-tree. (Note in particular that a linear tree clasper is always assumed to be simple.) 
The left-most and right-most leaves of $G$ in Figure~\ref{ckmove} are called the \emph{ends} of $G$.
As a convention, any simple $C_k$-tree is linear for $k\le 2$ ;  
the ends of a linear $C_1$-tree (resp. $C_2$-tree) are its two leaves (resp. a choice of any two leaves). 

Let $t$ be a linear $C_k$-tree with ends $f_0,f_k$. 
Since $t$ is a disk, we can travel from $f_0$ to $f_k$ along $\partial t$ so that 
we meet all other leaves $f_1,...,f_{k-1}$ in this order.
If $f_s$ is an $i_s$-leaf $(s=0,...,k)$, we can consider two vectors
$(i_0,...,i_k)$ and $(i_k,...,i_0)$ and may assume that 
$(i_0,...,i_k)\leq(i_k,...,i_0)$, where \lq$\leq$\rq~  is the lexicographic order in ${\Bbb Z}^{k+1}$.
We call $(i_0,...,i_k)$ the {\em ordered index} of $t$ and denote it by o-index$(t)$. 
In the following, we will simply denote by $(i_0...i_k)$ an ordered index $(i_0,...,i_k)$. 

By combining Lemma \ref{graph} and \cite[Lemma~5.1]{MYftisl}, we have the following lemma.
\begin{lemma} \label{o-index}
(1)~Let $t$ and $t'$ be two linear $C_k$-trees for $\1$ with same ordered index.  
Then either $\1_{t}\stackrel{C_{k+1}+c}{\sim} \1_{t'}$, or $\1_{t}\cdot \1_{t'}\stackrel{C_{k+1}+c}{\sim} \1$.

(2)~Let $t$ be a linear $C_k$-tree $(k\geq 3)$ for $\1$ with o-index$(t)=(i_0...i_k)$. 
If $i_0=i_1$ or $i_{k-1}=i_k$, then
$\1_{t}\stackrel{C_{k+1}+c}{\sim} \1$.
\end{lemma}

For each sequence $I=(i_0...i_k)$ of integers in $\{1,2\}$,
let $T(I)$ denote the \emph{choice} of a $2$-string link obtained from $\1_2$ by surgery along a linear 
$C_k$-tree with o-index $I$. 
Let also $T^*(I)$ denote the choice of a $2$-string link such that $T(I)\cdot T^*(I)\stackrel{C_{k+1}+c}{\sim} \1_2$. 
In particular, the $2$-string links $T(12)$, $T(121)$, $T(1221)$, $T(12221)$ and $T(21112)$
are chosen to be the string links $I$, $Y$, $H$, $S^1_{id}$ and $S^2_{id}$ introduced in Section \ref{sec:abelian}, respectively. 
We note that by Lemma~\ref{o-index}~(1), there are essentially two choices in $\mathcal{SL}(2) / (C_{k+1}+c)$ 
for each sequence $I$, namely $T(I)$ and $T^*(I)$. 

\subsection{Proofs of Proposition \ref{nonabeliansl} and Theorem \ref{abelian2}~(2) }\label{sec:proof0}

We first show why $\mathcal{SL}(n)/(C_3+c)$ is not abelian for any $n\ge 3$.

Since the $C_k$-concordance implies the $C_m$-concordance if $k>m$, 
this will imply statement (2) of Theorem \ref{abelian2}, and since the $C_k$-equivalence implies the $C_k$-concordance, 
we will also deduce Proposition \ref{nonabeliansl}.

Let $\sigma_1$ and $\sigma_2$ be the Artin generators for the $3$-braid group. 
Then $\sigma_1^2$, $\sigma_1^{-2}$, $\sigma_2^2$ and $\sigma_2^{-2}$ are $3$-string links. 
Since the closure of $\sigma_1^2\sigma_2^2\sigma_1^{-2}\sigma_2^{-2}$ is a copy of the Borromean rings,  
$\sigma_1^2\sigma_2^2\sigma_1^{-2}\sigma_2^{-2}$ has nontrivial Milnor invariant $\mu(123)$ of length 3, 
which is a $C_3$-concordance invariant. 
Hence $\sigma_1^2\sigma_2^2\sigma_1^{-2}\sigma_2^{-2}$ is not $C_3$-concordant to $\1_3$. 
This implies that $\mathcal{SL}(3)/(C_3+c)$ is not commutative : 
suppose indeed that $\sigma_1^2$ and $\sigma_2^2$ commute in $\mathcal{SL}(3)/(C_3+c)$, then 
$\sigma_1^2\sigma_2^2\sigma_1^{-2}\sigma_2^{-2}$ is $C_3$-concordant to $\sigma_2^2\sigma_1^2\sigma_1^{-2}\sigma_2^{-2}=\1_3$, 
which leads to a contradiction.  

\subsection{Abelian cases: $k\le 8$}\label{sec:abelianc}
We now show that $\mathcal{SL}(2)/(C_k+c)$ is abelian for $k\le 8$.

Since the $C_k$-concordance implies the $C_m$-concordance for $k>m$, it is sufficient to show that $\mathcal{SL}(2)/(C_8+c)$ is abelian.  

We use the same strategy as in Section \ref{sec:abelian}, where we showed that $\mathcal{SL}(2)/C_6$ is abelian. 
More precisely,  
we first chose generating sets for the successive quotients $\mathcal{SL}_{i}^c(2)/(C_{i+1}+c)$ ($i\le 7$), 
where $\mathcal{SL}_{i}^c(2)$ is the set of $(C_{i}+c)$-trivial $2$-string links,  
to obtain a set of generators for $\mathcal{SL}(2)/(C_8+c)$, then we show that any two of these generators commute. 
In this discussion, we may again ignore local generators, since they are central in $\mathcal{SL}(2)$.   

Let
\[\mathcal{H}^c_1=\{T(12)\},~\mathcal{H}^c_2=\{T(121)\},~\mathcal{H}^c_3=\{T(1221)\},~\mathcal{H}^c_4=\{T(12221), T(21112)\}\]
\[\textrm{and }\mathcal{H}^c_5=\{T(122221), T(211112), T(121221)\}. \]
In \cite{MYftisl}, the authors show that $\mathcal{H}^c_i$ is a generating set for $\mathcal{SL}_{i}^c(2)/(C_{i+1}+c)$, 
$i\le 5$.\footnote{
In \cite{MYftisl}, $T(121212)$ is chosen instead of $T(121221)$, but these two string links are $C_6$-equivalent by the AS relation.}  

Let us now pick a generating set for $\mathcal{SL}_{6}^c(2)/(C_{7}+c)$. 
By Lemma~\ref{o-index}~(2), a string link obtained from $\1_2$ by surgery along a linear $C_6$-tree with index 
$\{1,2^{(6)}\}$ or $\{1^{(6)},2\}$ is $C_7$-concordant to $\1_2$.  
So it is enough to consider $C_6$-trees with index $\{1^{(2)},2^{(5)}\}$, $\{1^{(3)},2^{(4)}\}$, $\{1^{(4)},2^{(3)}\}$ or $\{1^{(5)},2^{(2)}\}$. 
Furthermore, the IHX relation implies that it is sufficient to consider linear $C_6$-trees. 

By the IHX relation, the ends of a linear $C_6$-trees with index $\{1^{(2)},2^{(5)}\}$ 
can be chosen to be the two $1$-leaves, so that the only possible o-index is $(1222221)$. 
Similarly, we may assume that the ends of a linear $C_6$-trees with index $\{1^{(3)},2^{(4)}\}$ are both $2$-leaves.
Then by Lemma~\ref{o-index} ~(2), the o-index should be of the form $(21ijk12)$ for some $i,j,k\in\{1,2\}$. 
Since the index is $\{1^{(3)},2^{(4)}\}$, we have three possibilities, namely $ijk=212$, $122$, or $221$. 
By definition, the only two possible o-indices are then $(2112212)$ and $(2121212)$. 
Summarizing, we may assume that \\
 (1)~all linear $C_6$-trees with index $\{1^{(2)},2^{(5)}\}$ have o-index $(1222221)$,\\
 (2)~all linear $C_6$-trees with index $\{1^{(3)},2^{(4)}\}$ have o-index $(2121212)$ or $(2112212)$,\\
 (3)~all linear $C_6$-trees with index $\{1^{(4)},2^{(3)}\}$ have o-index $(1212121)$ or $(1211221)$,\\
 (4)~all linear $C_6$-trees with index $\{1^{(5)},2^{(2)}\}$ have o-index $(2111112)$.\\
\noindent (The last two case are deduced from the first two by exchanging 1 and 2.) 

The following lemma is useful to further reduce the number of generators.
\begin{lemma}\label{annoying}
For an integer $k\geq 5$, let $t$ and $t'$ be two linear $C_k$-trees for $\1_2$
whose respective o-indices are either of the form $(ijjiiI)$ and $(ijijiI)$,
where $I$ is a sequence of $k-4$ integers in $\{1,2\}$ and where $\{i,j\}=\{1,2\}$.
Then either $(\1_2)_{t}\stackrel{C_{k+1}+c}{\sim} (\1_2)_{t'}$, or $(\1_2)_{t}\cdot (\1_2)_{t'}\stackrel{C_{k+1}+c}{\sim} \1_2$.
\end{lemma}
\begin{proof}
Suppose that o-index$(t)$ is of the form $(ijjiiI)$.  
By Lemmas \ref{lem:STUc} and \ref{lem:twist}, we can assume without loss of generality that there a $3$-ball 
which intersects $\1_2\cup t$ as shown on the left-hand side of Figure \ref{fig:t}. 
Using Lemma \ref{lem:twist} again, we have that $(\1_2)_t$ is $C_{k+1}$-equivalent to $(\1_2)_{t'}$, 
where $t'$ is shown in the figure. 

By the IHX relation, $(\1_2)_{t'}$ is $C_{k+1}$-equivalent to 
$L\cdot S$, where $L$ and $S$ are string links as illustrated in Figure~\ref{fig:t}. 
Notice that, by Lemma \ref{lem:twist}, $L$ is $C_{k+1}$-equivalent 
to a string link obtained from $\1_2$ by surgery along a $C_k$-trees with o-index $(ijijiI)$. 
So, proving that $S$ is $(C_{k+1}+c)$-trivial would imply Lemma \ref{annoying}. 
\begin{figure}[!h]
   \includegraphics{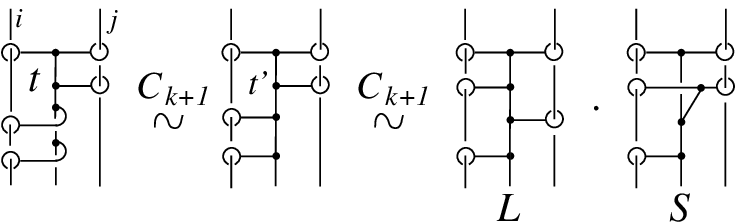}
   \caption{ }\label{fig:t}
  \end{figure}

\noindent We now show that $S$ is indeed $(C_{k+1}+c)$-trivial. 
The proof is given in Figure~\ref{fig:t'} as follows. 
By Lemma \ref{lem:twist} and the IHX relation, we have $S\stackrel{C_{k+1}}{\sim} U\cdot V$, where $U$ and $V$ are as shown, 
and by Lemmas~\ref{lem:cc} and \ref{lem:STUc}, $V$ is $C_{k+1}$-concordant to the string link $W$ represented on the right-hand side.
   \begin{figure}[!h]
   \includegraphics{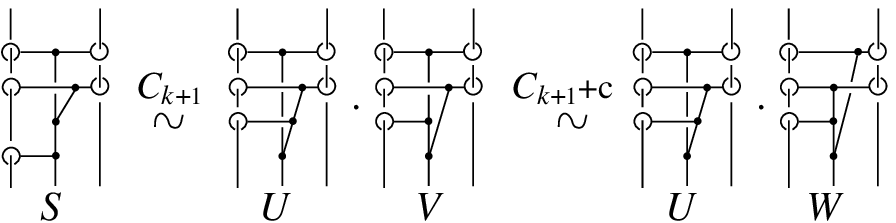}
   \caption{ }\label{fig:t'}
  \end{figure}

\noindent It then follows from the AS relation that $U\cdot W\stackrel{C_{k+1}}{\sim} \1_2$, which concludes the proof.
\end{proof}
\begin{remark}
The same result holds if  $t$ and $t'$ have respective o-indices $(Iiijji)$ and $(Iijiji)$. 
\end{remark}

Using this, we can assume that all linear $C_6$-trees with index $\{1^{(3)},2^{(4)}\}$ and $\{1^{(4)},2^{(3)}\}$ 
have o-index $(2121212)$ and $(1212121)$, respectively. It follows that 
\[\mathcal{H}^c_6=\{T(1222221),~T(1212121),~T(2121212), ~T(2111112)\}\] 
is a generating set for $\mathcal{SL}_{6}^c(2)/(C_{7}+c)$. 

Similarly, we have the following generating set for $\mathcal{SL}_{7}^c(2)/(C_{8}+c)$:
\[\mathcal{H}^c_7=\left\{
\begin{array}{l}
T(12222221),~T(21111112),~T(12111221),\\
T(21122212),~T(12211221),~T(12112221)
\end{array}\right\}.\]

In the following, we call \emph{generator of degree $k$} ($k\le 7$) any element $T(I)$ of the generating set 
$\mathcal{H}^c_k$ or its inverse $T^*(I)$ in $\mathcal{SL}_{k}^c(2)/(C_{k+1}+c)$.

In order to prove that $\mathcal{SL}(2)/(C_8+c)$ is abelian, it suffices to show that any two generators of degree $\le 7$ 
do commute in $\mathcal{SL}(2)/(C_8+c)$.

By Lemma~\ref{lem:I}, both $T(12)$ and $T^*(12)$ are central in $\mathcal{SL}(2)$. 
Moreover by Remark~\ref{rem:commute},  
two generators of degrees $k$ and $l$ commute in $\mathcal{SL}(2)/(C_8+c)$ if $k+l\geq 8$. 
Hence it is enough to check the commutativity of generators of degrees $k$ and $l$ for  
$(k,l)=(3,4),(2,5),(3,3),(2,4),(2,3),(2,2)$.

In order to prove the commutativity in the case $(3,4)$, we need the following. 
\begin{lemma}\label{o-index2}
Let $t$ be a linear $C_k$-tree $(k\geq 1)$ for $\1_n$ with $k$ odd.   
Then $(\1_n)_{t}$ is $C_{k+1}$-concordant 
to its image under orientation-reversal of all strings. 
\end{lemma}
A proof is easily obtained by combining Lemma~\ref{graph} and arguments similar to those in the proof of \cite[Lemma~5.1~(3)]{MYftisl}.
Here, let us only illustrate the general idea on an example.  
Consider the linear $C_7$-tree $t$ for $\1_2$ illustrated on the left-hand side of Figure \ref{flip}.  
Let also $t'$ be the linear $C_7$-tree for $\1_2$ illustrated on the right-hand side of the figure.  
(notice that both $t$ and $t'$ have o-index $(12112221)$.)
On one hand, the two string links obtained by surgery along $t$ and $t'$ 
are obtained from one another by reversing the orientation of all strings. 
On the other hand, by Lemma \ref{lem:twist} we have that $(\1_2)_{t}$ is $C_8$-equivalent to $(\1_2)_c$, 
where $c$ is the $C_7$-tree shown in the figure.   
Let $\tilde{c}$ be obtained by a $180$-degree rotation of $c$ around the axis $a$ fixing the leaves, see Figure \ref{flip}.  
 \begin{figure}[!h]
  \includegraphics[scale=0.85]{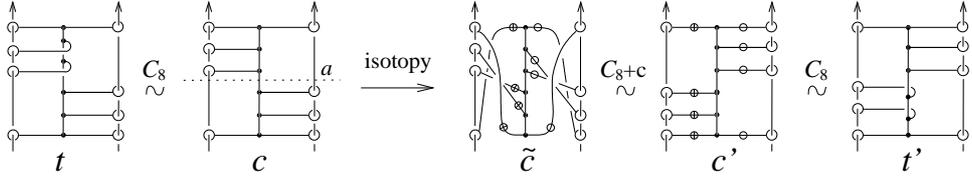}
  \caption{The two string links obtained by surgery along $t$ and $t'$ are $C_8$-concordant.} \label{flip}
 \end{figure} 
By sliding the leaves of $\tilde{c}$ repeatedly, we can deform it into the $C_7$-tree $c'$ shown in Figure \ref{flip}, 
which only differs from $t'$ by an even number of half-twists on its edges. 
By Lemmas \ref{lem:STUc} and \ref{lem:twist}, we obtain that 
$(\1_2)_{t}$ and $(\1_2)_{t'}$ are $C_8$-concordant.
(In the general case, the fact that the degree is odd ensures that there is an even number of half-twists.) 

We can now prove the desired commutativity property in the case $(k,l)=(3,4)$.

\medskip
{\bf Case $(3,4)$}:
We first show that $T(1221)$ and $T(12221)$ commute in $\mathcal{SL}(2)/(C_8+c)$.
Since both string links are symmetric, we note that $T(12221)\cdot T(1221)$ is obtained from $T(1221)\cdot T(12221)$ by 
orientation-reversal of both strings.
By Lemma~\ref{lem:slide}, $T(1221)\cdot T(12221)$ is $C_8$-equivalent to 
$(\1_2)_{t\cup s}\cdot L$, where $t$ and $s$ are tree claspers for $\1_2$ as illustrated in 
Figure \ref{fig:sym}, and where $L$ is a string link obtained from $\1_2$ by surgery along some $C_7$-trees. 
   \begin{figure}[!h]
   \includegraphics{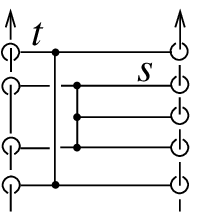}
   \caption{ }\label{fig:sym}
  \end{figure}
Since $L$ is $C_8$-concordant to a product $L'$ of generators of degree 7, we have that 
$T(1221)\cdot T(12221)$ is $C_8$-concordant to $(\1_2)_{t\cup s}\cdot L'$.
Hence $T(12221)\cdot T(1221)$ is $C_8$-concordant to the $2$-string link 
$(\1_2)_{t\cup s}\cdot L'$ with orientation reversed.  
On the other hand, by Lemma~\ref{lem:cc}, $(\1_2)_{t\cup s}$ is $C_8$-equivalent to its image under orientation-reversal, and 
by Lemma~\ref{o-index2}, each generator of degree 7 is $C_8$-concordant to its image under orientation-reversal.  
It follows that $T(1221)\cdot T(12221)$ and $T(12221)\cdot T(1221)$ are $C_8$-concordant.\\
\noindent The fact that $T(1221)$ and $T(21112)$ commute in $\mathcal{SL}(2)/(C_8+c)$ is shown completely similarly, 
as well as the other cases 
(namely, the commutativity of $T(1221)$ with $T^*(12221)$ and $T^*(21112)$, 
and the commutativity of $T^*(1221)$ with $T(12221)$, $T(21112)$, $T^*(12221)$ and $T^*(21112)$).  

Before we deal with the remaining cases, we make an observation. 
Since the group $\mathcal{SL}_{7}^c(2)/(C_{8}+c)$ is generated by the 6 elements of $\mathcal{H}^c_7$, 
and since there are 6 independent Milnor invariants of length $8$ \cite[Appendix B]{cochran}, 
we have that $\mathcal{SL}_{7}^c(2)/(C_{8}+c)$ is a free abelian group with rank 6 -- and in particular, has no $2$-torsion.
Moreover, by Claim~\ref{lem:ycentral},  
both $T(121)^2(=Y^2)$ and $T^*(121)^2(=(Y^*)^2)$ are $C_3$-equivalent to a central element in $\mathcal{SL}(2)$. 
So we get the following as an application of the `$C_k$-concordance version' of Lemma~\ref{lem:commute2}.
\begin{lemma}\label{lem:commute3}
Let $L$ be either $T(121)$ or $T^*(121)$. 
Suppose that $L$ satisfies the following conditions for some integers $l$ and $m$ ($l<m<8$) : 
\begin{itemize}
\item $L$ commutes with any $(C_l+c)$-trivial string link in $\mathcal{SL}(n)/(C_m+c)$, 
\item $L$ commutes with any $(C_m+c)$-trivial string link in $\mathcal{SL}(n)/(C_8+c)$, 
\end{itemize}
Suppose moreover that $(C_3+c)$-trivial and $(C_l+c)$-trivial string links commute in $\mathcal{SL}(n)/(C_8+c)$. 
Then 
$L$ commutes with any $(C_l+c)$-trivial string link in $\mathcal{SL}(n)/(C_{8}+c)$.  
\end{lemma}

We can now prove the desired commutativity property in the remaining cases $(k,l)=(2,5),(3,3),(2,4),(2,3),(2,2)$.
From now on, let $L$ be either $T(121)$ or $T^*(121)$. 

{\bf Case $(2,5)$}: 
By Remark~\ref{rem:commute}, $L$ commutes with any $C_5$-trivial string link in $\mathcal{SL}(2)/C_7$ 
and with any $C_7$-trivial string link in $\mathcal{SL}(2)/C_8$. Also, 
any $C_3$-trivial string link commutes with any $C_5$-trivial string link in $\mathcal{SL}(2)/C_8$. 
By applying Lemma~\ref{lem:commute3} for $l=5$ and $m=7$, we have that $L$ commutes with any $(C_5+c)$-trivial string link in
$\mathcal{SL}(2)/(C_8+c)$.

{\bf Case $(3,3)$}:  We already showed that 
$T(1221)$ commutes with any $(C_4+c)$-trivial string link (Case $(3,4)$).
Since $T(1221)$ is $C_3$-trivial and $T(1221)\cdot T^*(1221)$ is $C_4$-trivial,  
by Lemma~\ref{lem:commute1'}, we have that 
$T(1221)$ and $T^*(1221)$ commute in $\mathcal{SL}(2)/(C_8+c)$.

{\bf Case $(2,4)$}: 
Again by Case $(3,4)$, we know that 
any $(C_3+c)$-trivial string link commutes with any $(C_4+c)$-trivial string link in $\mathcal{SL}(2)/(C_8+c)$. 
By Remark~\ref{rem:commute}, 
$L$ commutes with any $C_4$-trivial string link in $\mathcal{SL}(2)/C_6$, and 
with any $C_6$-trivial string link in $\mathcal{SL}(2)/C_8$. 
By applying Lemma~\ref{lem:commute3} for $l=4$ and $m=6$, we thus have that 
$L$ commutes with any $(C_4+c)$-trivial string link in $\mathcal{SL}(2)/(C_8+c)$.

{\bf Case $(2,3)$}: We already showed in Case $(2,5)$ above that 
$L$ commutes with any $(C_5+c)$-trivial string link in $\mathcal{SL}(2)/(C_8+c)$. 
By Remark~\ref{rem:commute}, 
$L$ commutes with any $C_3$-trivial string link in $\mathcal{SL}(2)/C_5$. 
By Case $(3,3)$, any $C_3$-trivial string link commutes with any $C_3$-trivial string link in $\mathcal{SL}(2)/C_8$. 
By applying Lemma~\ref{lem:commute3} for $l=3$ and $m=5$, we have that 
 $L$ commutes with any $(C_3+c)$-trivial link in $\mathcal{SL}(2)/(C_8+c)$.

{\bf Case $(2,2)$}: 
Note that we already showed that 
$T(121)$ commutes with any $(C_3+c)$-trivial string link (Case $(2,3)$). 
Since $T(121)$ is $C_2$-trivial and $T(121)\cdot T^*(121)$ is $C_3$-trivial,  
we can use Lemma~\ref{lem:commute1'} to show that $T(121)$ and $T^*(121)$ commute in $\mathcal{SL}(2)/(C_8+c)$.

This concludes the proof that $\mathcal{SL}(2)/(C_k+c)$ is abelian for all $k\leq 8$.

\subsection{Nonabelian case}\label{sec:ruby}

Let $A$, $\overline{A}$, $B$ and $\overline{B}$ be 2-string links as illustrated in 
Figure~\ref{fig:abba}.
  \begin{figure}[!h]
\includegraphics[trim=0mm 0mm 0mm 0mm, width=.6\linewidth]{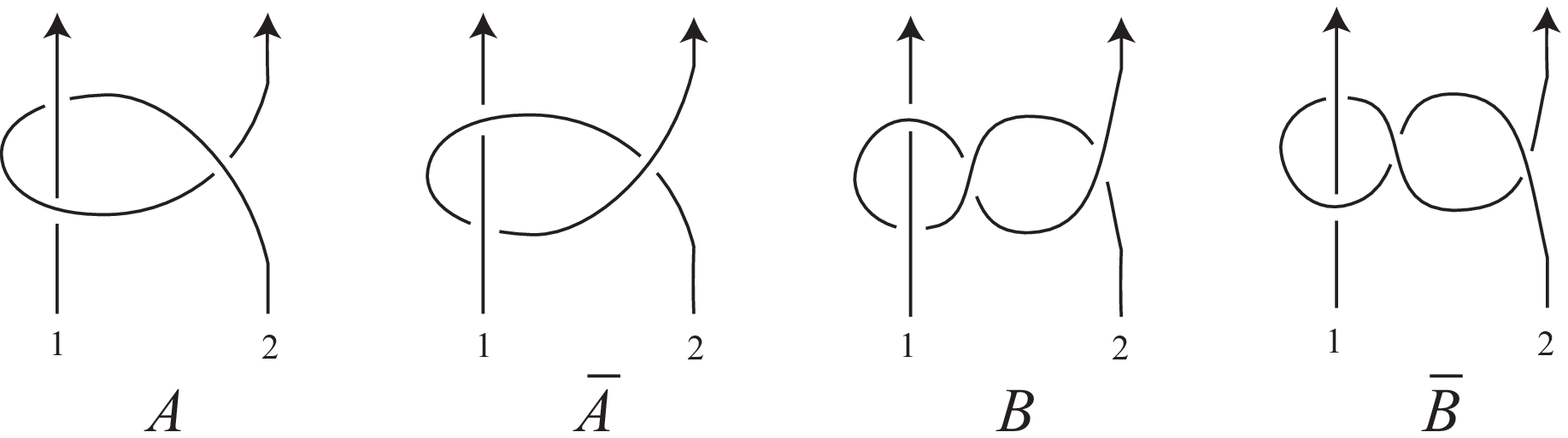}
  \caption{}\label{fig:abba}
 \end{figure}
The following implies that $\mathcal{SL}(2)/(C_k+c)$ is not abelian for all $k\ge 9$.
\begin{proposition}
The two string links $A$ and $B$ do not commute in $\mathcal{SL}(2)/(C_9+c)$.
\end{proposition}
\begin{proof}
The proof relies on a direct computation, using a program\footnote{
This program uses the programming language \texttt{Ruby}, and was written by Professor Tetsuji Kuboyama.} 
based on the algorithm given by Milnor in \cite{Milnor2} to compute $\mu$-invariants. 
We have that the Magnus expansion of the $1st$ longitude of $AB\overline{A}\overline{B}$ is 
\[\begin{array}{l}
1
-XXXYXYYY+XXXYYYXY+5XXYXXYYY \\
-6XXYXYXYY-XXYXYYYX+XXYXYYYY+6XXYYXYXY \\
-5XXYYXYYY-5XXYYYXXY+XXYYYXYX+5XXYYYXYY \\
-XXYYYYXY-5XYXXXYYY+9XYXXYXYY-9XYXXYYXY \\
+5XYXXYYYX-6XYXYXYYX+6XYXYXYYY+9XYXYYXXY \\
-9XYXYYXYY-XYXYYYXX+XYXYYYYX-9XYYXYXXY \\
+6XYYXYXYX+9XYYXYYXY-5XYYXYYYX+5XYYYXXXY \\
-5XYYYXXYX+XYYYXYXX-6XYYYXYXY+5XYYYXYYX \\
-XYYYYXYX+YXXXXYYY+6YXXXYYXY-5YXXXYYYX \\
-YXXXYYYY-9YXXYXXYY+9YXXYXYYX-9YXXYYXYX \\
+9YXXYYXYY+5YXXYYYXX-6YXXYYYXY+6YXYXXXYY \\
-6YXYXXYYY-6YXYXYYXX+6YXYXYYYX-6YXYYXXXY \\
+9YXYYXXYX-9YXYYXYYX-YXYYYXXX+6YXYYYXXY \\
+YXYYYYXX-6YYXXXYXY+5YYXXXYYY+9YYXXYXXY \\
-9YYXXYXYY-9YYXYXXYX+9YYXYXXYY+6YYXYXYXX \\
-9YYXYYXXY+9YYXYYXYX-5YYXYYYXX-YYYXXXXY \\
+5YYYXXXYX-5YYYXXXYY-5YYYXXYXX+6YYYXXYXY \\
+YYYXYXXX-6YYYXYXYX+5YYYXYYXX+YYYYXXXY \\
-YYYYXYXX+(\text{higher degree terms}),
\end{array}\] 
\noindent where the Magnus expansion is defined here by sending
the first meridian to $1+X$ and the second meridian to $1+Y$. 
Hence some Milnor invariants of length 9 do not vanish on the string link $AB\overline{A}\overline{B}$. 
This implies that 
$AB\overline{A}\overline{B}$ is not $C_9$-concordant to $\1_2$. 
Now, if $A$ commutes with $B$ in $\mathcal{SL}(2)/(C_9+c)$, 
then $AB\overline{A}\overline{B}$ is $C_9$-concordant to 
$BA\overline{A}\overline{B}$, which is concordant to $\1_2$. 
This is a contradiction.
\end{proof}

\appendix

\section{Proofs of Corollaries \ref{cor1} and \ref{cor2}}\label{app}

In this short appendix, we briefly outline the proofs of Corollaries \ref{cor1} and \ref{cor2}. 
That is, we show how the abelian group struncture on the set of $C_k$-equivalence classes 
(resp. $C_k$-concordance classes) of $2$-string links implies the Goussarov-Habiro conjecture at the corresponding degree. 
As explained in the introduction, this is merely an adaptation of Habiro's argument for proving Theorem \ref{cnknots}, so 
we do not reproduce here all the details of this proof, but give precise references to the technical results from Habiro's paper 
that are used in this arguments, and emphasize the role of the commutativity of $\mathcal{SL}(n)/C_{k+1}$.

For simplicity, we only give the arguments for Corrollary \ref{cor1}. 
They are easily adapted to the case of $C_k$-concordance and finite type concordance invariants to obtain Corrollary \ref{cor2}. 

Actually, Habiro's proof of Theorem \ref{cnknots} can be adapted to show the following.
\begin{proposition}\label{propapp}
Suppose that the group of $C_{k+1}$-equivalence classes of $n$-string links is abelian. Then 
two $n$-string links cannot be distinguished by any finite type invariant of order $\le k$ if and only if they 
are $C_{k+1}$-equivalent. 
\end{proposition}
\begin{proof}
The fact that two $C_{k+1}$-equivalent string links share all finite type invariants of degree up to $k$ is well known \cite[Cor. 6.8]{H}, 
so we only need to show the converse implication.  

Supppose that the group $\mathcal{SL}(n)/C_{k+1}$ is abelian. 
Following Habiro's strategy for proving Theorem \ref{cnknots} (Theorem 6.18 in \cite{H}), we consider the homomorphism of abelian groups 
(from an additive to a multiplicative abelian group) 
 $$ \varphi_k: \mathbb{Z}\mathcal{SL}(n) \longrightarrow \mathcal{SL}(n)/C_{k+1}$$ 
which maps an $n$-string link $L$ to its $C_{k+1}$-equivalence class $[L]_{k+1}$.  
The key point is that the existence (ie the well-definedness) of this map relies on the fact that $\mathcal{SL}(n)/C_{k+1}$ is abelian. 
Notice that $\varphi_k$ is indeed additive, since for two $n$-string links $L$ and $L'$, we have 
$$ \varphi_k(L\cdot L'-L-L') 
     = [L\cdot L']_{k+1}\cdot [L]^{-1}_{k+1}\cdot [L']^{-1}_{k+1}
     = [\1]_{k+1} = \varphi_k(\1). $$
Now, we have that $\varphi_k$ is a finite type invariant of degree $k$. 
This is proved by Habiro in the (string-)knot case in \cite[Prop. 6.16]{H}. 
His argument relies on a deep result on the structure of the Goussarov-Vassiliev filtration \cite[Prop. 6.10]{H}, 
which involves advanced clasper theory. But Habiro actually established the latter result not only for (string-)knots, but also for string links (and more generally for surface string links). 
So we can freely use \cite[Prop. 6.10]{H} to show that $\varphi_k$ is a finite type invariant of degree $k$. 
The proof of Proposition \ref{propapp} is then completed by simply following \cite[Thm. 6.18]{H} as follows. 
If two $n$-string links $L$ and $L'$ cannot be distinguished by any finite type invariant of order $\le k$, 
then $\varphi_k(L)=\varphi_k(L')$. This implies that $[L]_{k+1} = [L']_{k+1}$, that is, $L$ and $L'$ are $C_{k+1}$-equivalent. 
\end{proof}

\end{document}